\documentclass[a4paper,12pt,twoside]{article}
\usepackage[english]{babel}
\usepackage{amsthm}
\usepackage{amsmath}
\usepackage[arrow, matrix, curve]{xy}
\usepackage{mathrsfs}
\usepackage[utf8]{inputenc}
\usepackage{amssymb}
\usepackage{xcolor}
\usepackage{stmaryrd}
\usepackage[T1]{fontenc}
\usepackage{hyperref}    

\newtheorem{thm}{Proposition}[section]
\newtheorem{prop}[thm]{Proposition}
\newtheorem{theorem}[thm]{Theorem}
\newtheorem{lemma}[thm]{Lemma}

\newtheorem{theorem*}[]{Theorem}

\theoremstyle{definition}
\newtheorem{defn}[thm]{Definition}
\newtheorem{rem}[thm]{Remark}
\newtheorem{exmp}[thm]{Example}

\usepackage {color}
\definecolor{yellow}{rgb}{1,1,0}

\usepackage{listings}

\DeclareMathOperator{\BTtor}{BT_{tor}}

\DeclareMathOperator{\BK}{BK}

\DeclareMathOperator{\Max}{Max}

\DeclareMathOperator{\M}{\mathbb{M}}

\DeclareMathOperator{\coker}{coker}
\DeclareMathOperator{\BKe}{BK_1}
\DeclareMathOperator{\BKn}{BK_n}
\DeclareMathOperator{\BTn}{BT_n}
\DeclareMathOperator{\BTe}{BT_1}

\DeclareMathOperator{\BKtor}{BK_{tor}}

\DeclareMathOperator{\Lie}{Lie}

\DeclareMathOperator{\BT}{BT}

\DeclareMathOperator{\hig}{h}
\DeclareMathOperator{\Tor}{Tor}
\DeclareMathOperator{\Spec}{Spec}

\newcommand{\e}{\'{e} }

\newcommand{\Wn}{W_n }
\newcommand{\xio}{\xi_0}

\newcommand{\psio}{\overline{\psi}}
\newcommand{\varphio}{\overline{\varphi}}

\newcommand{\pb}{p^\flat}
\newcommand{\Zp}{\mathbb{Z}_p}

\newcommand{\N}{\mathbb{N}}

\newcommand{\Z}{\mathbb{Z}}

\newcommand{\Qp}{\mathbb{Q}_p}

\newcommand{\OC}{\mathcal{O}_C}
\newcommand{\OCb}{\mathcal{O}_C^\flat}

\newcommand{\im}{\mathrm{im}}

\newcommand{\id}{\mathrm{id}}

\begin{document}

\begin{center}
\Large A Note on Dieudonn\e Theory over Perfectoid Rings\\
\vspace{0.5cm}
\normalsize Timo Henkel
\footnote{Department of Mathematics, Technische Universität Darmstadt, Schloßgartenstr. 7, 64289 Darmstadt, Germany, thenkel@mathematik.tu-darmstadt.de}
\vspace{0.2cm}
\end{center} 
\noindent \textbf{Abstract}. For a perfectoid ring $R$ and a natural number $n$ we investigate the essential image of the category of truncated by $n$ Barsotti-Tate groups under the anti-equivalence between commutative, finite, locally free, $R$-group schemes of $p$-power order and torsion Breuil-Kisin-Fargues modules over $R$. We describe the associated semi-liner algebra data and show as a consequence that every $\BTn$-group over $R$ is the $p^n$-torsion of some $\BT$-group.

\begin{center}
Introduction
\end{center}

For $k$ being a perfect field of characteristic $p$, classical Dieudonn\e theory provides a semi-linear algebra description of truncated by $n$ Barsotti-Tate groups. This can be used to show that every such group arises as the $p^n$-torsion subgroup of some Barsotti-Tate group. 
Lau  extended these results to perfect rings of characteristic $p$ in \cite{L2}, reproving unpublished results of Gabber by a different method. Recently, in \cite{AB}, Anschütz and Le Bras generalized the existing instances of Dieudonn\e theory to so called Prismatic Dieudonn\e Theory, including results on Dieudonn\e theory for perfectoid rings, which were partially proven by Lau before in \cite{L1}. 
In this note, we use the results of \cite{AB} to classify $\BTn$-groups over perfectoid rings in terms of semi-linear algebra and investigate the consequences for lifting properties of $\BTn$-groups to $\BT$-groups in this setting.

%For $k$ being a perfect field of characteristic $p$  with Dieudonn\e theory at hand, to show that every $\BTn$-group can be lifted to some Barsotti-Tate group over $k$. These results have also been extended to perfect rings of characteristic $p$ in \cite{L2}. Recently, in \cite{AB}, Anschütz and Le Bras extended the existing instances of Dieudonn\e theory to so called prismatic Dieudonn\e theory, which in particular provides Dieudonn\e theory for perfectoid rings.  In this note, we investigate the consequences for lifting properties of $\BTn$-groups in this setting. 
 The semi-linear algebra objects associated with $\BTn$-groups are the following:
A $\BK_n$-module over a perfectoid ring $R=W(S)/ \xi$ is a triple $(M,\varphi,\psi)$, where $M$ is a finite projective $W_n(S)$-module, $\varphi \colon M^\sigma \to M$ and $\psi \colon M \to M^\sigma$ are $W_n(S)$-linear maps such that $\varphi \circ \psi = \xi$ and $ \psi \circ \varphi = \xi$. 
In the case $n=1$ one additionally demands that $\coker(\varphi)$ is finite projective as an $S/ \xi$-module and that the induced sequence $M/ \xi  \xrightarrow{\overline{\psi}} M^\sigma / \xi \xrightarrow{\overline{\varphi}} M/ \xi$ is exact. 
The category of $\BKn$-modules over $R$, denoted by $\BKn(R)$, also admits a fibrewise characterization: 
A torsion Breuil-Kisin-Fargues module over $R$ is a $\BKn$-module if and only if it is killed by $p^n$ and all fibres are truncated Dieudonn\e modules of level $n$. Combining this fact with the fibrewise description of $\BTn$-groups and the properties of the functor of \cite{AB}, which classifies commutative finite locally free $p$-groups, we show the following result:

\begin{theorem*}
Let $R$ be a perfectoid ring. For all $n \in \N $ there is an anti-equivalence of categories 
\begin{gather*}
\BTn(R) \cong \BKn(R).
\end{gather*}
\end{theorem*}

We apply lifting arguments to obtain the following corollary:
\begin{theorem*}
Let $R$ be a perfectoid ring and $n \in \N$. The truncation functor 
\begin{gather*}
\BT(R) \to \BTn(R)
\end{gather*}
is essentially surjective, i.e. every $\BTn$-group over $R$ arises as the $p^n$-torsion subgroup of some Barsotti-Tate group over $R$.
\end{theorem*}

In the first chapter we recall the definition and general facts about perfectoid rings. Over such a ring $R$ we briefly recall the relevant categorical anti-equivalence $\M$ of \cite{AB} which is the starting point of our investigation (cf. \ref{equivalence for BTtor}). In the next chapter we turn attention to the category of $\BTe$-groups and $\BKe$-modules. After proofing the fibrewise characterization (cf. \ref{characterization of BKe}) we establish Theorem 1 for $n=1$ (cf. \ref{BTe=BKe}). In the last chapter, we explain analogous results for $\BTn$-groups (cf. \ref{BTn=BKn}) and finally show that, with this description at hand, we can use lifting of normal representations to proof Theorem 2 (cf \ref{BT to BTn surjective}).

\vspace{0.5 cm}

\textit{Acknowledgements:}
The author thanks E. Lau and T. Wedhorn for helpful discussions.

\tableofcontents

\vspace{2cm}
 Notations and assumptions:
\begin{enumerate}
\item[$\cdot$] Once and for all we fix some prime number $p$.
\item[$\cdot$] All rings are commutative with $1$.
\item[$\cdot$] For a ring $R$ with $pR=0$ we denote by $W(R)$ the ring of $p$-typical Witt vectors of $R$ and for $n \in \N$ by $W_n(R)$ the ring of $p$-typical Witt vectors truncated by $n$. 
\item[$\cdot$] For a ring $R$ with $pR=0$ we denote the Frobenius morphism $x \mapsto x^p$ by $\sigma$. By abuse of notation we also denote the induced Frobenius morphism on $W(S)$ and on $W_n(S)$ by $\sigma$. 
\end{enumerate}

\section{Prerequisits}

We recall some general facts about perfectoid rings and briefly state the anti-equivalences obtained in \cite{L1} (§9 and §10), \cite{AB} (§ 4 and §5.1) and \cite{SW1} (Theorem 17.5.2) which characterize Barsotti-Tate groups and $p$-groups over perfectoid rings in terms of semilinear algebra.
\subsection{Generalities on perfectoid rings}

We use the definition of perfectoid rings as quotients of perfect prisms as used in \cite{L1}.

\begin{defn}
\begin{enumerate}
\item A ring $R$ with $pR=0$ is called \emph{perfect} if the Frobenius morphism $\sigma \colon R \to R, \  x \mapsto x^p$ is bijective.
\item For a perfect ring $S$ an element $\xi=(\xi_0,\xi_1,...) \in W(S)$ is called \emph{distinguished} if $\xi_1 \in S^\times$ and $S$ is $\xi_0$-adically complete.
\item A ring $R$ is called \emph{perfectoid} if there is an isomorphism $R \cong W(S) / \xi$, where $S$ is a perfect ring and $\xi \in S$ is a distinguished element.
\end{enumerate}
\end{defn}

\begin{rem}
\begin{enumerate}
\item
A ring $R$ is perfectoid if and only if it is $\pi$-adically complete for some element $\pi \in R$ such that $\pi^p$ divides $p$, the Frobenius map $\varphi \colon R/p \to R/p $ is surjective and the kernel of $\theta \colon W(R^\flat) \to R$ is principal, say generated by an element $\xi$ (here $\theta$ denotes the unique lift of the natural surjection $R^\flat := \lim\limits_{\leftarrow \sigma} R/p \to R/p$).
 In this case we obtain an isomorphism $W(R^\flat)/\xi \to R$ which is induced by $\theta$ and whenever we have a presentation of a perfectoid ring $R=W(S)/\xi$ we may identify $S$ with $R^\flat$ (cf. \cite{L1} Remark 8.6).
\item The notion of a perfectoid ring defined above is sometimes called \emph{integral perfectoid ring} to distinguish from the notion of a \emph{perfectoid Tate Huber ring}, which plays the central role in the theory of perfectoid spaces. However, given a perfectoid Tate Huber pair $(R,R^+)$, the ring of integral elements $R^+$ is a perfectoid ring in the sense defined above.

\item For a perfectoid ring $R= W(S)/ \xi$ we have $W_n(S)/ \xi = R/p^n$ for all $n \in \N$ and the elements $p$ and $\xi \in W(S)$ are nonzerodivisors. The latter implies that $R$ is $p$-torsion free if and only if $W_n(S)$ is $\xi$-torsion free for all $n \in \N$. Indeed, the $W(S)$-module 
\begin{gather*}
\lbrace(x,y) \in W(S)^2 \ | \ \xi x =p^ny \rbrace / \lbrace (p^nz , \xi z)  \ | \ z \in W(S) \rbrace
\end{gather*}
is isomorphic to both $\lbrace x \in \Wn(S) \ | \ \xi x =0 \rbrace $ and $\lbrace y \in R \ | \  p^n y =0 \rbrace$ via $(x,y) \mapsto x$ and $(x,y) \mapsto y$ respectively.
\item Every perfectoid ring $R=W(S)/ \xi$ is $p$-adically complete. In particular, $\Max(R)=\Max(R/p)=\Max(S/\xio) = \Max(S)$.
\end{enumerate}
\end{rem}

\begin{exmp}
\begin{enumerate}
\item Every perfect ring $S$ may be written as $W(S)/p$ and is hence a perfectoid ring. In particular, every perfect field of characteristic $p$ is a perfectoid ring.
\item Let $C$ be an algebraically closed complete non-archimedean extension of $\Qp$ with ring of integers $\OC$. Choosing a compatible system of $p$-power roots of $p$ gives an element $p^\flat$ of $\OC^\flat$ such that $\OC^\flat$ is $p^\flat$-adically complete. 
We obtain an isomorphism $W(\OC^\flat)/([p^\flat]-p) \cong \OC$ which depicts $\OC$ as a perfectoid ring.
\end{enumerate}
\end{exmp}

\subsection{Short review of classification results}
Let $R=W(S)/\xi$ be a perfectoid ring. We denote the category of $p$-divisible groups over $R$ by $\BT(R)$.

\begin{defn}
\begin{enumerate}
\item A  \emph{(minuscule) Breuil-Kisin-Fargues module} for $R$ is a triple $(M,\varphi, \psi)$, where $M$ is a finite projective $W(S)$-module and $\varphi \colon M^\sigma \to M$ and $\psi \colon M \to M^\sigma$ are $W(S)$-linear maps such that $\psi \circ \varphi= \xi$ and $\varphi \circ \psi= \xi$.
\item
We denote the category of Breuil-Kisin-Fargues modules for $R$ by $
\BK(R)$ (the morphisms are $W(S)$-linear maps which are compatible with the $\varphi$'s; the compatibility with $\psi$'s is then automatic as explained below) .
\end{enumerate}
\end{defn}

Note that since $\xi \in W(S)$ is a nonzerodivisor, the maps $\varphi$ and $\psi$ are injective and hence $\psi$ is determined by $\varphi$ (and vice versa) via the formula $\psi(x)= \varphi^{-1}(\xi x)$. We only add the datum to draw the connection to other semi linear algebra data as defined below.

\begin{theorem}\label{equivalence for BT}
There is an exact, compatible with base change, anti-equivalence 
\begin{gather*}
\BT(R) \cong \BK(R), \\ G \mapsto \M(G),
\end{gather*}
admitting an exact, compatible with base change, inverse. 
\end{theorem}
\begin{proof}
\cite{SW1} Theorem 17.5.2. resp. \cite{AB} Corollary 4.3.8.
\end{proof}

\begin{defn}
Let $\BTtor(R)$ denote the category of commutative, finite, locally free $R$-group schemes which are killed by a power of $p$. 
\end{defn}

Note that the category $\BTtor(R)$ is denoted by $\text{pGR}(\Spec R)$ in \cite{L1}. We justify our notation by Theorem \ref{equivalence for BTtor} and and Remark \ref{relation between the two functors} below.

\vspace{0.5cm}

The following definition is taken from \cite{L1} Definition 10.5. Such objects are introduced as \emph{torsion prismatic Dieudonn\e modules} in \cite{AB} Definition 5.1.1. 

\begin{defn}
\begin{enumerate}
\item A \emph{torsion Breuil-Kisin-Fargues module for $R$} is a triple $M=(M,\varphi,\psi)$, where $M$ is a finitely presented $W(S)$-module of projective dimension $\leq 1$, killed by a power of $p$, and $M^\sigma \xrightarrow{\varphi} M , M \xrightarrow{\psi} M^\sigma$
are $W(S)$-linear maps such that $\varphi \circ \psi= \xi$ and $\psi \circ \varphi= \xi$.
\item We denote the category of torsion Breuil-Kisin-Fargues modules for $R$ by $\BKtor(R)$ (the morphisms are $W(S)$-linear maps which are compatible with the $\varphi$'s and the $\psi$'s).
\end{enumerate}
\end{defn}

\begin{rem}\label{torsion free case}
If $R$ is $p$-torsion free (equivalently $W_n(S)$ is $\xi$-torsion free for all $n \in \N$) for any object $M=(M,\varphi,\psi)$ of $\BKtor(R)$ the module $M$ is $\xi$-torsion free and hence $\varphi$ and $\psi$ are injective. So $\varphi$ determines $\psi \colon M \to M^\sigma$ via the formula $\psi(x)=\varphi^{-1}(\xi x)$ (and vice versa). Indeed, assume $p^nM=0$ then $M$ fits into a short exact sequence 
\begin{gather*}
 0 \to P_1 \to P_0 \to M \to 0,
 \end{gather*} 
 where $P_1$ and $P_0$ are finite projective $W(S)$-modules s.t. $p^nP_0 \subseteq P_1$. We obtain an injection
 \begin{gather*}
  p^nP_0/p^nP_1 \to P_1 /p^nP_1 
 \end{gather*}
 of $W_n(S)$-modules, where the latter is projective and hence $M = p^nP_0/p^nP_1$ is torsion free. But $\xi\in W_n(S)$ is a nonzerodivisor by assumption.
\end{rem}

 By abuse of notation, we use the same symbol $\M$ for the anti-equivalences in Theorem \ref{equivalence for BT} and Theorem \ref{equivalence for BTtor} below. 

\begin{theorem}\label{equivalence for BTtor}
There is an exact, compatible with base change, anti-equivalence
\begin{gather*}
\BTtor(R) \cong \BKtor(R), \\ G \mapsto \M(G),
\end{gather*}
admitting an exact, compatible with base change, inverse.
\end{theorem}
\begin{proof}
\cite{AB} Theorem 5.1.4.
\end{proof}

\begin{rem}\label{relation between the two functors}
\begin{enumerate}
\item Zariski-locally on $\Spec(R)$ every $G \in \BTtor(R)$ can be written as the kernel of an isogeny of $p$-divisible groups over $R$ , so Zariski-locally, $G=\ker(H_1 \to H_2)$ (cf. \cite{BBM} Theorem 3.1.1). The functor of the theorem above can Zariski-locally on $R$ be defined as $\M(G):=\coker(\M(H_1) \to \M(H_2))$ using the anti-equivalence of Theorem \ref{equivalence for BT} (cf. proof of \cite{L1} Theorem 10.12 and proof of \cite{AB} Theorem 5.1.4).
\item
By \cite{AB} Proposition 4.3.2 and the explicit description of $\M$ given in Theorem 5.1.4, the anti-equivalence $\M$ extends crystalline Dieudonn\e Theory of \cite{BBM}. 
\end{enumerate}
\end{rem}

\section{$\BTe$-groups and Dieudonn\e spaces}

The goal of this section is to determine the essential image of the full subcategory $\BTe(R)$ of $\BTtor(R)$, defined below, for a perfectoid ring $R$ under the anti-equivalence $\BTtor(R) \cong \BKtor(R)$. Fix a perfectoid ring $R=W(S)/\xi$.

\subsection{Truncated by $1$ Barsotti-Tate groups}

\begin{defn}
Let $\BT_1(R)$ denote the full subcategory of $\BTtor(R)$ consisting of objects $G$ such that $G$ is annihilated by $p$ and for $G_\circ:=G \otimes_R R/p$ 
the sequence
\begin{gather*}
  G_\circ \xrightarrow{F} G_\circ^{(p)} \xrightarrow{V} G_\circ 
  \end{gather*}   
 is exact ($F$ resp. $V$ denote the Frobenius resp. Verschiebung of $G_\circ$). Objects of $\BTe(R)$ are called \emph{truncated by $1$ Barsotti-Tate groups over $R$} or \emph{$\BT_1$-groups over $R$}.
\end{defn}

\begin{defn} Let $G =\Spec(A) $ be a $\BTe$-group over $R$. 
\begin{enumerate}
\item
The locally constant function $ \hig \colon \Spec(R/p) \to \N_0$ which is uniquely determined by $s \mapsto \log_p( \dim_{\kappa(s)} (A \otimes_R \kappa(s)))$, for $s \in \Max(R/p)$ with residue field $\kappa(s)$, is said to be the \emph{height of $G$}.
\item
The locally constant function $ \dim \colon \Spec(R/p) \to \N_0$ which is uniquely determined by $s \mapsto  \dim_{\kappa(s)} (\Lie(G \otimes_R \kappa(s)))$, for $s \in \Max(R/p)$ with residue field $\kappa(s)$, is said to be the \emph{dimension of $G$}.
\end{enumerate} 
\end{defn}

\begin{exmp}
 Let $G$ be a $p$-divisible group over $R$ of constant height $h$ and dimension $d$, and let $G[p]$ denote its group of $p$-torsion. Then $G[p]$ is a $\BTe$-group of height $h$ and dimension $d$. As a corollary of the lifting results of §3, we will see that every object of $\BTe(R)$ is of this form.
\end{exmp}

\begin{rem}
\begin{enumerate}
\item
Let $G$ be some $\BTe$-group over $R$ with $M = \M(G)$ the associated $\BKtor$-module under the anti-equivalence of Theorem \ref{equivalence for BTtor}.
Since $\M$ commutes with base change and extends the Dieudonn\e Theory of \cite{BBM}, given some point $s \in \Max(S)$ with residue field $\kappa$ (which is perfect and of characteristic $p$), the base change $M_s=(M_s:=M\otimes_S \kappa,\varphi_s:=\varphi \otimes \id_\kappa, \psi_s:= \psi \otimes \id_\kappa)$ of $M$ to $\kappa$ is a \emph{Dieudonn\e space} over $\kappa$, i.e. $M_s$ is a finite dimensional $\kappa$-vector space such that $\psi_s \circ \varphi_s=0$, $\varphi_s \circ \psi_s=0$ and the sequence
\begin{gather*}
 M_s \xrightarrow{\psi_s} M_s^\sigma \xrightarrow{\varphi_s} M_s
 \end{gather*} 
 is exact (note that this implies that $ M_s^\sigma \xrightarrow{\varphi_s} M_s \xrightarrow{\psi_s} M_s^\sigma$ is exact as well).
 \item The \emph{height} of a Dieudonn\e space $(M,\varphi,\psi)$ over a perfect field $k$ of characteristic $p$ is defined to be the $k$-dimension of $M$. The $k$-dimension of $\coker(\varphi)$ is called \emph{dimension of $M$}.
 \item For a perfect field $k$ of characteristic $p$ we denote the category of Dieudonn\e spaces over $k$ by $\BKe(k)$. The functor $\BTe(k) \to \BKe(k)$ which is induced by $\M$ is a height and dimension preserving anti-equivalence, since the Dieudonn\e Theory of \cite{BBM} §3 provides this result and $\M$ extends this theory. 
 \end{enumerate}
\end{rem}

The following lemma gives a fibrewise characterization of $\BTe$-groups. This will lead to a description of the essential image of $\BTe(R)$ under the anti-equivalence of Theorem \ref{equivalence for BTtor}.

\begin{lemma}\label{characterization of BTe}
Let $G \in \BTtor(R)$ such that $G$ is killed by $p$. Then $G \in \BTe(R)$ if and only if $G \otimes_R \kappa(s) \in \BTe(\kappa(s))$ for all $s \in \Max(R/p)$.
\end{lemma}
\begin{proof}
Noting that $\Max(R)= \Max(R/p)$ this is a corollary of the following general assertion on finite locally free group schemes.
\end{proof}

\begin{lemma}\label{exact sequences for finite locally free group schemes}
Let $X$ be an affine scheme. A sequence of finite locally free group schemes 
\begin{gather*}
G' \xrightarrow{\varphi} G \xrightarrow{\psi} G''
\end{gather*}
over $X$
is exact if and only if $\psi \circ \varphi=0$ and for every closed point $x \in X$ the sequence $
G'_x \to G_x \to G''_x$ 
is exact.
\end{lemma}
\begin{proof}
\cite{AJ} Proposition 1.1.
\end{proof}

\subsection{Modules associated with $\BTe$-groups}

The fibrewise characterization of $\BTe(R)$ hints that its essential image consists of such $M \in \BKtor(R)$ which satisfy $pM=0$ and whose fibres are Dieudonn\e spaces for all closed points of $\Spec(S)$. The goal of this section is to show that this essential image is also given via the following global definition.

\begin{defn}\label{Defn of BK1}
Let $\BKe(R)$ be the full subcategory of $\BKtor(R)$ whose objects are given by triples $M=(M, \varphi, \psi)$, such that

\begin{enumerate}
\item[(a)] $M$ is killed by $p$ (and hence is a finite projective $S$-module by Lemma \ref{proj over S}),
\item[(b)] $\coker(\varphi)$ is finite projective as an $S/\xio$-module,
\item[(c)] the induced sequence 
\begin{gather*}
  M / \xio \xrightarrow{\overline{\psi}} M^\sigma/ \xio \xrightarrow{\overline{\varphi}} M / \xio
\end{gather*}
is exact.
\end{enumerate}
\end{defn}
 Objects of $\BKe(R)$ are also called $\BKe$-modules over $R$.
The lemma below ensures that the underlying $S$-module of a $\BKe$-module is in fact finite projective as indicated in (a):

\begin{lemma}\label{proj over S}
Let $M$ be a finite $W(S)$-module which is killed by $p$. Then $M$ is of projective dimension $1$ as a $W(S)$-module if and only if it is finite projective as an $S$-module.
\end{lemma}
\begin{proof}
 \cite{BS} Lemma 7.8. 
\end{proof}

\begin{rem}
\begin{enumerate}
\item If $R=k$ is a perfect field of characteristic $p$, we can write $k=W(k)/p$. Hence $S=k$ and $\xio=0$ and in this case and we recover the definition of Dieudonn\e spaces over $k$. 
\item More generally, if $R=S$ is a perfect ring, a $\BKe$-module is the same as a \emph{truncated Dieudonn\e module of level $1$ over $R$} as defined in \cite{L2} Definition 6.1.
\item If $R$ is $p$-torsion free, an object of $\BKtor$ which is killed by $p$ automatically satisfies property (c) of the definition above. 
\end{enumerate}
\end{rem}

\begin{exmp}
\begin{enumerate}
\item Let $\OC = W(\OC^\flat)/([p^\flat]-p)$ be the ring of integers of some algebraically closed complete non archimedean extension of $\Qp$ with tilt $\OC^\flat$. In this case we have $\xio=p^\flat$ and a $\BKtor$-module over $\OC$ which is killed by $p$ is uniquely determined by a pair $(M,\varphi)$, where $M$ is a finite free $\OC^\flat$-module and $\varphi \colon M^\sigma \to M$ is a linear map whose cokernel is killed by $p^\flat$. 
\item For every $a \in \OCb$ such that $|\pb| \leq |a| \leq 1$ the cokernel of the linear map $(\OCb)^\sigma \to \OCb, x \mapsto a \sigma^{-1}(x)$ is killed by $\pb$. 
However, its cokernel is finite projective as a $\OCb / \pb \OCb$-module if and only if it is isomorphic to $(\OCb/ \pb \OCb)^n$ as an $\OCb/ \pb \OCb$-module for some $n \in \N$ because $\OCb/ \pb \OCb$ is local. But this is satisfied if and only if  $|a| \in \lbrace |\pb|, 1 \rbrace$ since multiplication with $a$ is not equal to $0$ on $(\OCb/ \pb \OCb)^n$ if $|a|>| \pb|$. Hence, up to isomorphism, we only obtain two different $\BKe$-modules in this case.
 One associated with the $\BTe$-group $\underline{\Z/p\Z}_{\OC}$ and one associated with $\mu_{p,\OC}$. This observation matches our expectation that every $\BKe$-module over $\OC$ can be lifted to some $\BK$-module over $\OC$, because up to isomorphism there are only two $\BK$-modules over $\OC$. Those are by the anti-equivalence of Theorem \ref{equivalence for BT} given by the $\BT$-groups $\underline{\Qp / \Zp}$ and $\mu_{p^\infty}$. 
\end{enumerate}
\end{exmp}

We also have a notion of height and dimension for $\BKe$-modules:

\begin{defn}
Let $M=(M,\varphi,\psi) \in \BKe(R)$.
\begin{enumerate}
\item
The locally constant function $ \hig \colon \Spec(S/\xio) \to \Z$ uniquely determined by $s \mapsto \dim_{\kappa(s)} (M \otimes_S \kappa(s))$, for $s \in \Max(S/\xio)$ with residue field $\kappa(s)$, is said to be the \emph{height of $M$}.
\item
The locally constant function $ \dim \colon \Spec(S/\xio) \to \Z$ uniquely determined by $s \mapsto \dim_{\kappa(s)}( \coker(\varphi) \otimes_{S/\xio} \kappa(s))$, for $s \in \Max(S/\xio)$ with residue field $\kappa(s)$, is said to be the \emph{dimension of $M$}.
\end{enumerate}
\end{defn}

The following lemma is the analogue of Lemma \ref{exact sequences for finite locally free group schemes} providing a fibrewise characterization of $\BTe(R)$:

\begin{lemma}\label{fibrewise for BTe}
Let $A$ be a ring and 
\begin{gather*}
M^\bullet=(M' \xrightarrow{\varphi} M \xrightarrow{ \psi} M' \xrightarrow{ \varphi} M)
\end{gather*}
be a sequence of finite projective $A$-modules such that $ \varphi \circ \psi =0$ and $ \psi \circ \varphi=0$. The following assertions are equivalent:
\begin{enumerate}
\item $M^\bullet \otimes_A \kappa(x)$ is exact for all $x \in \Max(A)$,
\item $M^\bullet$ is exact and $\coker(\varphi)$ is a finite projective $A$-module.
\item $M \xrightarrow{ \psi} M' \xrightarrow{ \varphi} M$ is exact and $\coker(\varphi)$ is a finite projective $A$-module.
\end{enumerate}   In this case $\im(\psi)$ and $ \ker(\varphi)$ are direct summands of $M'$, $\im(\varphi)$ and $\ker(\psi)$ are direct summands of $M$ and $\coker(\psi)$ is a finite projective $A$-modules as well.
\end{lemma}

\begin{proof}
Obviously, 2. implies 3.
If $M \xrightarrow{ \psi} M' \xrightarrow{ \varphi} M$  is exact and $\coker(\varphi)$ is finite projective, $\im(\varphi), \ker(\varphi),\coker(\psi),$ $\im(\psi)$ and $\ker(\psi)$ satisfy the additional assertion of the lemma. Hence 1. follows. So lets assume 1. and consider the sequence
\begin{gather*}
0 \to \im (\varphi) \to M  \to \coker(\varphi) \to 0.
\end{gather*} To show that $\coker ( \varphi)$ is a projective $A$-module, we establish that $ \im (\varphi) \otimes_{A} \kappa(x) \to M \otimes_{A} \kappa(x)$ is injective for each $x \in \Max(A)$. By assumption, we have an exact sequence of $\kappa(x)$-vector spaces 
\begin{gather*}
M \otimes_A \kappa(x) \xrightarrow{\psi_x} M' \otimes_A \kappa(x) \xrightarrow{\varphi_x} M \otimes_A \kappa(x).
\end{gather*}
Here $\varphi_x$ factors as
\begin{gather*}
M'  \otimes_{A} \kappa(x) \xrightarrow{\pi} \im (\varphi) \otimes_{A} \kappa(x) \to  M  \otimes_{A} \kappa(x).
\end{gather*}
Now let $s \in \im(\varphi) \otimes_{A} \kappa(x)$ which is mapped to $0$. By surjectivity of $\pi$ we can choose $y \in M' \otimes_A \kappa(x)$ with $\pi(y)=s$. Then $\varphi_x(y)=0$, hence we find $z \in M \otimes_A \kappa(x)$ with $\psi_x(z)=y$. Therefore $s=\pi(\psi_x(z))=0$ as desired.
 Note that we can do the same calculations for $\psi$ instead of $\varphi$ to obtain projectivity of $\coker(\psi)$. Consequently, $\im(\psi), \ker(\psi), \im(\varphi)$ and $\ker(\varphi)$ satisfy the additional assertion of the lemma.
The exactness of $M^\bullet$ then follows by applying Nakayama's lemma.
\end{proof}

\begin{rem}
Let $R=W(S)/ \xi \to R'=W(S')/\xi'$ be a morphism of perfectoid rings. Then the image of $\xi$ under the induced morphism $W(S) \to W(S')$ is a distinguished element of $W(S')$ and we can identify $R'=W(S')/ \xi$. By the previous lemma we obtain a base change functor $\BTe(R) \to \BTe(R')$.
\end{rem}

\begin{prop}\label{characterization of BKe}
Let $M=(M,\varphi,\psi) \in \BKtor(R)$ such that $pM=0$. Then $M \in \BKe(R)$ if and only if $M \otimes_S \kappa(s) \in \BKe(\kappa(s))$ for all $s \in \Max(S/\xio)$. Moreover in this case the sequence 
\begin{gather*}
M^\sigma/ \xio  \xrightarrow{\varphio} M/ \xio \xrightarrow{\psio} M^\sigma / \xio
\end{gather*}
is exact as well.
\end{prop}
\begin{proof}
This follows from Lemma \ref{fibrewise for BTe}.
\end{proof}

\begin{prop}\label{BTe=BKe}
The anti-equivalence $\M$ of Theorem \ref{equivalence for BTtor} induces a height and dimension preserving anti-equivalence 
\begin{gather*}
\BTe(R) \cong \BKe(R).
\end{gather*}
\end{prop}

\begin{proof}
Let $G$ be a $\BTtor$-group over $R$ and $M=\M(G)$ be the associated $\BKtor$-module over $R$.
Since $\M$ is exact and has an exact inverse, $G$ is annihilated by $p$ if and only if $M$ is annihilated by $p$. Since $\M$ and an inverse are compatible with base change and induce height and dimension preserving anti-equivalences $\BTe(k) \cong \BKe(k)$ for every perfect field $k$, the proposition follows from Lemma \ref{characterization of BTe} and Lemma \ref{characterization of BKe}.
\end{proof}

\section{Generalization to $\BTn$ and lifting properties}

We still fix a perfectoid ring $R=W(S)/ \xi$. Similar to the previous section, we determine the essential image of $\BTn$-groups under the anti-equivalence $\M$ of Theorem \ref{equivalence for BTtor}. 
At the end we apply a lifting argument to the obtained semilinear algebra data to show that over $R$ every $\BTn$-group arises as the $p^n$-torsion of some $\BT$-group.

\subsection{Modules associated with $\BTn$-groups}

\begin{defn}
For $n \geq 2$ we denote by $\BTn(R)$ the full subcategory of $\BTtor(R)$ whose objects are those groups $G$ which are annihilated by $p^n$ and such that the sequence
\begin{gather*}
G \xrightarrow{\cdot p^{n-1}} G \xrightarrow{\cdot p} G 
\end{gather*}
is exact.
\end{defn}

As in the case for $n=1$, we have a fibrewise description of the category $\BKn(R)$:

\begin{lemma}\label{characterization of BTn}
Let $G \in \BTtor(R)$ and $n \in \N$ such that $G$ is annihilated by $p^n$. Then $G$ is an object of $\BTn(R)$ if and only if $G \otimes_R \kappa(s)$ is an object of $\BT_n(\kappa(s))$ for all $s \in \Max(R/p)$.
\end{lemma}
\begin{proof}
Note that $\Max(R) = \Max(R/p)$ and use Lemma \ref{exact sequences for finite locally free group schemes}.
\end{proof}

\begin{rem}\label{shorten}
Let $n \geq m $ be two natural numbers and $G \in \BTn(R)$. The $p^m$-torsion subgroup $G[p^m]$ of $G$ is an object of $\BT_m(R)$.
\end{rem}

Now we define the full subcategory of $\BKtor(R)$ which will be shown to correspond to $\BTn(R)$ under the anti-equivalence of Theorem \ref{equivalence for BTtor}.

\begin{defn}
For $n \geq 2$ we denote by $\BKn(R)$ the full subcategory of $\BKtor(R)$ consisting of objects $M=(M,\varphi,\psi)$ such that $M$ is killed by $p^n$ and finite projective as a $\Wn(S)$-module.
\end{defn}

Objects of $\BKn(R)$ are also called $\BKn$-modules over $R$ and if $R=S$ is perfect, those are the same as \emph{truncated Dieudonn\e modules of rank $n$ over $R$} as defined in \cite{L2} Definition 6.1.
\vspace{0.5cm}

If $R=S$ is perfect, so in particular if $R$ is a perfect field of characteristic $p$, we have the following classification result:

\begin{prop}\label{BTn=BKn R perfekt}
Assume that $R$ is perfect. Then the anti-equivalence $\M \colon \BTtor(R) \to \BKtor(R)$ of Theorem \ref{equivalence for BTtor} induces an anti-equivalence
\begin{gather*}
\BTn(R) \cong \BKn(R)
\end{gather*}
for all $n \geq 1$.
\end{prop}
\begin{proof}
Since $\M$ extends the Dieudonn\e theory of \cite{BBM}, this is the content of \cite{L2} Theorem 6.4.
\end{proof}

\begin{rem} 
Let $R=W(S)/ \xi \to R'=W(S') / \xi$ be a morphism of perfectoid rings. We obtain a base change functor $\BKn(R) \to \BKn(R')$. 
\end{rem}

As in the $n=1$ case, we also have the following fibrewise description:

\begin{lemma}\label{characterization of BKn}
Let $M=(M,\varphi,\psi) \in \BKtor(R)$ such that $p^nM=0$. Then $M \in \BKn(R)$ if and only if $M \otimes_{W_n(S)} W_n(\kappa(s)) \in \BKn(\kappa(s))$ for all $s \in \Max(R/p)$.
\end{lemma}

\begin{proof} One direction is clear. Conversely, assume that all fibres of $M=(M,\varphi,\psi)$ are $\BKn$-modules. We have to show that $M$ is a finite projective $W_n(S)$-module.
The assumption $M \otimes_{W_n(S)} W_n(\kappa(s)) \in \BT_n(\kappa(s))$ implies that $M/p \otimes_S \kappa(s) = (M \otimes_{W_n(S)} W_n(\kappa(s)))/p $ is an object of $ \BTe(\kappa(s))$ for all $s \in \Max(R/p)$ by \cite{L2} Lemma 6.2. This shows that $M/p \otimes_S \kappa(s)$ is an object of $\BKe(\kappa(s))$. Consequently, $M/p \in \BKe(R)$ by Lemma \ref{characterization of BKe} and $M/p$ is a finite projective $S$-module. By the local flatness criterion we are reduced to show that 
\begin{gather*}
\Tor_1^{W_n(S)}(S,M)=0.
\end{gather*}
This is the case if and only if the sequence
\begin{gather*}
M \xrightarrow{ \cdot p^{n-1}} M \xrightarrow{ \cdot p} M
\end{gather*}
is exact. But since the anti-equivalence $\M$ (of Theorem \ref{equivalence for BTtor}) and an inverse are exact and commute with base change, this is equivalent to 
\begin{gather*}
M\otimes_{W_n(S)} W_n(\kappa(s)) \xrightarrow{\cdot p^{n-1}} M\otimes_{W_n(S)} W_n(\kappa(s)) \xrightarrow{ \cdot p} M\otimes_{W_n(S)} W_n(\kappa(s))
\end{gather*}
being exact for all $s \in \Max(W_n(S))$ by Lemma \ref{exact sequences for finite locally free group schemes}. This is true because $M\otimes_{W_n(S)} W_n(\kappa(s))$ is a finite projective $W_n(\kappa(s))$-module by assumption.
\end{proof}

The fibrewise description then leads to

\begin{prop}\label{BTn=BKn}
The anti-equivalence $\M \colon \BTtor(R) \cong \BKtor(R)$ of Theorem \ref{equivalence for BTtor}, induces an anti-equivalence $\BTn(R) \cong \BKn(R)$ for all $n \geq 1$.
\end{prop}
\begin{proof}
This follows formally from Lemma \ref{characterization of BTn} and Lemma \ref{characterization of BKn} using the exactness and base change properties of the anti-equivalence $\M$ and the classification in the perfect case as stated in Proposition \ref{BTn=BKn R perfekt}. 
\end{proof}

\subsection{Normal representations for $\BKn$-modules}
The classification of $\BTn(R)$ via $\BKn$-modules at hand, we show that any $\BTn$-group can be lifted to a $\BT$-group.
We fix some $n \in \N \cup \lbrace \infty \rbrace$  and use the notation $W_\infty (R): = W(R), \ \BK_\infty (R) := \BK(R)$ and $ \BT_\infty := \BT(R)$ for a fixed perfectoid ring $R=W(S) / \xi$.

\begin{defn}
A \emph{normal representation} for a $\BKn$-module $(M,\varphi,\psi)$ over $R$ is a triple $(L,P,\Phi)$ where $L \subseteq M$ and $P \subseteq M$ are direct summands with $M=L \oplus P$ and $\Phi \colon M \to M^\sigma$ is a $W_n(S)$-linear bijection, such that, with respect to the decompositions $M=L \oplus P$ and $M^\sigma = \Phi(L)\oplus \Phi(P)$, we have $\psi = \alpha \oplus \xi \beta$ and $\varphi = \xi \alpha^{-1} \oplus \beta^{-1}$, where $\alpha = \Phi|_L \colon L \cong \Phi(L)$ and $\beta = \Phi|_P  \colon P \cong  \Phi(P)$.
\end{defn}

\begin{rem}
Every triple $(L,P,\Phi)$, where $P$ and $L$ are finite projective $\Wn(S)$-modules and $\Phi \colon P \oplus L \cong (P \oplus L)^\sigma$ is an isomorphism, determines a unique $\BKn$-module  over $R$, which we call the \emph{$\BKn$-module associated with $(L,P,\Phi)$}.
\end{rem}

The key ingredient to lifting $\BKn$-modules is the fact that every $\BKn$-module admits a normal representation. We need the following version of Nakayama's lemma for the proof.

\begin{lemma}\label{Nakayama for proj Modules}
Let $A$ be a ring and let $g \colon M \to N$ be a morphism of $A$-modules with $M$ finitely generated and $N$ finite projective. In both of the following cases $g$ is an isomorphism:
\begin{enumerate}
\item
$g \otimes_A \kappa(x) \colon M \otimes_A \kappa(x) \to N \otimes_A \kappa(x)$
is an isomorphism for all closed points $x \in \Spec(A)$.
\item $g \otimes_A A/I \colon M/IM \to N/IN$ is an isomorphism for some ideal $I\subseteq A$ contained in the Jacobson radical of $A$.
\end{enumerate}
\end{lemma}
\begin{proof}
2. follows from 1. which in turn is \cite{AB} Lemma 4.5.3.
\end{proof}

\begin{lemma}
Every $\BKn$-module $(M,\varphi,\psi)$ over $R$ admits a normal representation.
\end{lemma}
\begin{proof}
Let $M=(M, \varphi, \psi) \in \BKn(R)$ and $I := \ker(\Wn(S) \to S/ \xio) = (\xi,p)$ which is contained in the Jacobson radical of $\Wn(S)$. The pair $(\Wn(S),I)$ even is Henselian since the pairs $(\Wn(S),(p))$ and $(S,(\xi))$ have this property. 
Consider $\overline{M}=( \overline{M}, \overline{\varphi}, \overline{\psi})$, the base change of $M$ to $S/ \xio$. Choose submodules $\overline{L} \subseteq \overline{M}$ and $\overline{Q} \subseteq \overline{M}^\sigma$ such that $\overline{M}=  \overline{L} \oplus \im(\overline{\varphi})$ and $\overline{M}^\sigma = \im(\overline{\psi}) \oplus \overline{Q}$.
We have isomorphisms of $S/ \xio$-modules $\overline{\alpha} \colon \overline{L} \to \im(\overline{\psi})$ and $\overline{\beta} \colon \im(\overline{\varphi}) \to \overline{Q} $ which are induced by  $\overline{\psi}$ and $\overline{\varphi}^{-1}$  respectively.
  We claim that we can find a direct summand $Q \subseteq M^\sigma$ such that $Q/ I Q = \overline{Q}$.
   Indeed, let $Q'$ be any finite projective $W_n(S)$-module such that $Q'/ I Q' = \overline{Q}$, which exists since the pair $(\Wn(S),I)$ is Henselian.
    We can lift the maps $\overline{Q} \to \overline{M}^\sigma \to \overline{Q}$ to maps $Q' \to M^\sigma \to Q'$ whose composition is an isomorphism by Nakayama. 
    Then $Q \subseteq M^\sigma$, the image of the injection $Q' \to M^\sigma$, does the trick.
    The same argument shows that we can find a direct summand $L \subseteq M$ such that $L/ IL = \overline{L}$. Consider the commutative diagram
\begin{gather*}
 \begin{xy}
  \xymatrix{
 Q \ar@{->>}[r]^{\varphi} \ar@{->>}[d]^{/ I} & \varphi(Q) \ar@^{(->}[r] \ar[d] & M \ar@{->>}[d]^{/ I} \\
\overline{Q} \ar[r]^{\overline{\beta}^{-1}} & \im(\overline{\varphi}) \ar@^{(->}[r]  & \overline{M} .
  }
\end{xy}
 \end{gather*}
 Since $\overline{\beta}^{-1}$ is bijective, we see that $\varphi(Q)/ (I \varphi(Q)) = \im(\overline{\varphi})$. Now consider the map $\varphi(Q) \oplus L \to M$. 
 This is an isomorphism modulo $I$ and hence itself an isomorphism by Lemma \ref{Nakayama for proj Modules}. In particular, $\varphi(Q)$ is a finite projective $\Wn(S)$-module. 
Moreover,  $\varphi|_Q \colon Q \to \varphi(Q)$ is an isomorphism since it is an isomorphism modulo $I$.
  We set $\beta:= ( \varphi|_Q)^{-1}) \colon \varphi(Q) \to Q$.
   The same arguments show that $Q \oplus \psi(L)=M^\sigma$ and that $\alpha := \psi|_L \colon L \to \psi(L)$ is bijective. 
   Now define $\Phi \colon M= L \oplus \varphi(Q) \to \psi(L) \oplus Q = M^\sigma$ to be $\alpha \oplus \beta$. 
    By construction, the triple  $(L, \varphi(Q),\Phi)$ then is a normal representation for $M$.  
\end{proof}

\begin{prop}\label{BK to BKn surjective}
For $n \geq m$ ($m,n \in \N \cup \lbrace \infty \rbrace$) the truncation functor $\BKn(R) \to \BK_m(R)$ is essentially surjective.
\end{prop}
\begin{proof}
Let $M $ be  a $ \BK_m$-module over $R$ with normal decomposition $(L,P,\Phi)$. We  choose finite projective $W_n(S)$-modules $\tilde{L}$ and $\tilde{P}$ such that $\tilde{L}/p^m =L$ and $\tilde{P}/p^m = P$. We set $\tilde{M}:=\tilde{L} \oplus \tilde{P}$. 
Moreover, we can lift the map $\tilde{M} \to M \xrightarrow{\Phi}M^\sigma$ along $\tilde{M}^\sigma \to M^\sigma$ to obtain an isomorphism $\tilde{\Phi} \colon \tilde{M} \to \tilde{M}^\sigma$. Let $\tilde{M}=(\tilde{M},\tilde{\varphi},\tilde{\psi})$ be the $\BKn$-module over $R$ associated with $(\tilde{L}, \tilde{P}, \tilde{\Phi})$. 
Then $\tilde{M}/ p^m = M$ as desired.
\end{proof}
Now setting $n= \infty$ and using Proposition \ref{BTn=BKn} and Theorem \ref{equivalence for BT}, we obtain 

\begin{theorem}\label{BT to BTn surjective}
Every $\BTn$-group $G$ over $R$ is of the form $H[p^n]$ for some $\BT$-group $H$ over $R$.
\end{theorem}

\end{document}